\documentclass[11pt]{amsart}
\usepackage{amsmath,amsfonts,amssymb,amsthm,amscd,latexsym,euscript,verbatim}
\usepackage[all]{xy}

\makeatletter
\def\section{\@startsection{section}{1}%
  \z@{1.1\linespacing\@plus\linespacing}{.8\linespacing}%
  {\normalfont\Large\scshape\centering}}
\makeatother

\theoremstyle{plain}

\newtheorem*{conj*}{Root Groups Conjecture}
\newtheorem*{thm1.2}{(1.2) Theorem}
\newtheorem*{thm1.3}{(1.3) Theorem}
\newtheorem*{thm1.4}{(1.4) Theorem}
\newtheorem*{prop*}{Proposition}

\newtheorem{prop}{Proposition}[section]

\newtheorem{thm}[prop]{Theorem}

\newtheorem{lemma}[prop]{Lemma}
\newtheorem{facts}[prop]{Facts}
\newtheorem{hyps}[prop]{Hypotheses}

\theoremstyle{definition}

\newtheorem*{Def*}{Definition}
\newtheorem{Defs}[prop]{Definitions}

\newtheorem*{notation*}{Notation}
\newtheorem{remark}[prop]{Remark}
\newtheorem{remarks}[prop]{Remarks}
\newtheorem*{remark*}{Remark}

\newcommand{\ff}{\mathbb{F}}

\newcommand{\one}{{\bf 1}}

\numberwithin{equation}{section}

\begin{document}
\title[characterization of the Octonions]{A short characterization of the Octonions}

\author[Kleinfeld and Segev]{Erwin Kleinfeld\qquad Yoav Segev}
\address{Erwin Kleinfeld \\
1555 N.~Sierra St.~Apt 120, Reno, NV 89503-1719, USA}
\email{erwinkleinfeld@gmail.com}

\address{Yoav Segev \\
         Department of Mathematics \\
        Ben-Gurion University \\
        Beer-Sheva 84105 \\
         Israel}
\email{yoavs@math.bgu.ac.il}

\begin{abstract}
In this paper we
prove that if $R$ is
a proper alternative ring whose additive group has no $3$-torsion
and whose non-zero commutators are not zero-divisors, then $R$
has no zero-divisors.  It follows from a theorem of Bruck and Kleinfeld
that if, in addition, the characteristic of $R$ is not $2,$ then
the central quotient of $R$ is an octonion division algebra over some field.
We include other characterizations of octonion division algebras
and we also deal with the case where $(R,+)$ has $3$-torsion.
\end{abstract}

\date{\today}
\maketitle

%
\section{Introduction}

In this paper we characterize  proper (i.e.~not associative) alternative rings 
whose central quotient (in the sense of \cite{BK}) is an octonion division algebra over some field,
as in Theorems \ref{thm 1.1}.  Theorem  \ref{thm 1.3} is a consequence
of Theorem \ref{thm 1.1} and additional observations.
This characterization differs from previous characterizations as we do not emphasize ideals.
For historical information on the octonions see \cite{Bl, E},
and for their connection with algebraic groups see \cite{SV}.

The second author would like to thank the first
author for contacting him, and for including him in this project.

Before we state our main results, we draw the
attention of the reader to Remark \ref{rem 3.4} with
regards to the notion of a zero divisor. 

Hypothesis (2) of the following theorem may
seem awkward.  However, it holds under
very natural condition, i.e., under one of the hypotheses \ref{hyps 1.2}(T1)
or \ref{hyps 1.2}(T2) below.

\begin{thm}\label{thm 1.1}
Let $R$ be an alternative not associative ring.
Assume that
\begin{enumerate}
\item
A non-zero  commutator in $R$  is not a divisor of zero in $R$;

\item
if $x\in R,$ and $y,z$ are in the commutative center of $R,$
then $(xy)z=x(yz).$
\end{enumerate}
Then 
\begin{itemize}
\item[(a)]
$R$ contains no divisors of zero.

\item[(b)]
If, in addition, the characteristic of $R$ is not $2,$
then the central quotient ring $R//C$
is an $8$-dimensional octonion division algebra
over its center--the fraction field of the
center $C$ of $R.$
\end{itemize}
\end{thm}

We note that part (b) of Theorem \ref{thm 1.1}
is a consequence of part (a) and \cite[Theorem A]{BK}.
The notion of ``central quotient ring'' comes from \cite{BK}.

Consider now the following hypotheses.

\begin{hyps}\label{hyps 1.2}
$R$ is an alternative, not associative ring.
\begin{itemize}
\item[(H1)]
A non-zero  commutator in $R$  is not a divisor of zero in $R.$
\item[(H2)]
\begin{enumerate}
\item
no non-zero commutator of $R$ is nilpotent;

\item
non-zero elements of the nucleus of $R$ are not divisors of zero in $R.$
\end{enumerate}

\item[(H3)]
\begin{enumerate}
\item
no non-zero commutator of $R$ is nilpotent;

\item
$R$ has an identity;

\item
the nucleus of $R$ is a simple ring.
\end{enumerate}

\item[(T1)]
$3x=0\implies x=0,$ for all $x\in R.$

\item[(T2)]
\begin{enumerate}
\item
Commutative subrings of $R$ are associative;

\item
$2x=0\implies x=0,$ for all $x\in R.$
\end{enumerate}
\end{itemize}
\end{hyps}

As a consequence of Theorem \ref{thm 1.1} we prove
the following (multi-) theorem.

\begin{thm}\label{thm 1.3}
Let $R$ be an alternative not associative ring.
Assume that
\begin{enumerate}
\item
At least one of the hypotheses \ref{hyps 1.2}(H1), \ref{hyps 1.2}(H2), \ref{hyps 1.2}(H3) holds  in $R.$

\item
At least one of the hypotheses
\ref{hyps 1.2}(T1), \ref{hyps 1.2}(T2)
holds in $R.$
\end{enumerate}
Then
\begin{itemize}
\item[(a)]
$R$ contains no divisors of zero.

\item[(b)]
If, in addition, the characteristic of $R$ is not $2,$
then the central quotient ring $R//C$
is an $8$-dimensional octonion division algebra
over its center--the fraction field of the
center $C$ of $R.$
\end{itemize}
\end{thm}

Indeed, Theorem \ref{thm 1.3} is an immediate
consequence of Theorem \ref{thm 1.1} and the following
proposition.

\begin{prop}\label{prop 1.4}
Let $R$ be an alternative, not associative ring.

\begin{enumerate}
\item
If $R$ satisfies one of the hypotheses  \ref{hyps 1.2}(H2) or  \ref{hyps 1.2}(H3),
then $R$ satisfies hypothesis (1) of Theorem \ref{thm 1.1}.

\item
If $R$ satisfies hypothesis \ref{hyps 1.2}(H3), then the center of $R$ is a field.

\item
If $R$ satisfies one of the hypotheses  \ref{hyps 1.2}(T1) or  \ref{hyps 1.2}(T2),
then $R$ satisfies hypothesis (2) of Theorem \ref{thm 1.1}.
\end{enumerate} 
\end{prop}
\begin{proof}
See the end of section 2.
\end{proof}

\begin{remarks}\label{rem 1.5}
\begin{enumerate}
\item
The exclusion of characteristic $2$ in part (b) of Theorem \ref{thm 1.1}
is because the situation in this characteristic is different,
and \cite{BK} does not deal with characteristic $2.$

\item
In a commutative alternative ring 
$(x,y,z)^3=0,$ for each 
associator $(x,y,z)$ (see \cite[Lemma 9, p.~706]{S}).
Hence the hypothesis that commutative subrings are
associative is necessary if we want to prove 
that $R$ has no zero divisors.  Indeed, by
\cite{GR}, there are commutative
alternative algebras over $\ff_3$ which are not associative.
In particular, Theorem \ref{thm 1.1} is false once
hypothesis (2) is removed.

\item
{\bf Question.} Is it always true
in any alternative ring $R,$ that if $x\in R$ and $y,z$
are in the commutative center of $R$ (i.e.~$x,y,z$
are as in hypothesis 
(2) of Theorem \ref{thm 1.1}), then the
subring of $R$ generated by $x,y,z$ is commutative? 
\end{enumerate}
\end{remarks}

%
\section{Preliminaries on alternative rings}

Our main references for alternative rings are \cite{K1, K2}.
Let $R$ be a ring, not necessarily with $\one$ and not
necessarily associative.  

\begin{Defs}
Let $x,y,z\in R.$
\begin{enumerate}
\item
The {\bf associator} $(x,y,z)$ is defined to be
\[
(x,y,z)=(xy)z-x(yz).
\]

\item
The {\bf commutator} $(x,y)$ is defined to be
\[
(x,y)=xy-yx.
\]

\item
$R$ is an {\bf alternative ring} if
\[
(x,y,y)=0=(y,y,x),
\]
for all $x,y\in R.$

\item
The {\bf nucleus} of $R$ is denoted $N$ and defined by
\[
N=\{n\in R\mid (n,R,R)=0\}.
\]
Note that in an alternative ring the associator is skew symmetric in its $3$
variables (\cite[Lemma 1]{K2}).  Hence   $(R,n,R)=(R,R,n)=0,$ for $n\in N.$

\item
The {\bf center} of $R$ is denoted $C$ and defined
\[
C=\{c\in N\mid (c,R)=0\}.
\]

\item
The {\bf commutative center} of $R$ is denoted here by $Z$ and defined by
\[
Z=\{z\in R\mid (z,R)=0\}.
\]
\end{enumerate}
\end{Defs}

In the remainder of this paper {\bf $R$ is an alternative ring}
which is {\bf not associative}. 
$N$ denotes the nucleus of $R,$ $C$ its center and $Z$ its commutative center.

\begin{facts}\label{facts}
Let $R$ be an alternative ring.

\begin{enumerate}
\item
If $v\in R$ is a commutator, then $v^4\in N.$ 

\item
If $v\in R$ is a commutator, then $[v^2,R,R]v=0.$

\item
Any subring of $R$ generated by two elements is associative.
\end{enumerate}
\end{facts}
\begin{proof}
For (1)\&(2) see \cite[Theorem 3.1]{K1}. Part (3) is a theorem of E.~Artin.
\end{proof}

\begin{lemma}\label{lem Z a subring}
\begin{enumerate}
\item
Let $x,y\in R,$ with $(x,y)=0,$ and let $z\in Z.$
Then $xz$ commutes with $y.$

\item
The commutative center $Z$ of $R$ is a subring of $R.$

\item
If $3x=0\implies x=0,$ for all $x\in R,$ then $Z=C.$
\end{enumerate}
\end{lemma}
\begin{proof}
(1)\quad
By \cite[equation (6), p.~131]{K2} we have
\begin{equation}\label{eq x,y,r}
(xy,r)-x(y,r)-(x,r)y=3(x,y,r),\qquad\text{for all $x,y,r\in R$.}
\end{equation}
Taking $r=z,$ we get
\[
3(x,y,z)=0.
\]
But also
\[
(zx,y)=(zx,y)-z(x,y)-(z,y)x=3(z,x,y)=3(x,y,z)=0.
\]
\medskip

\noindent
(2)\quad
By (1), if $x,z\in Z$ and $y\in R,$ then $xz$ commute with $y,$
so $xz\in Z,$ and $Z$ is closed under multiplication.  Hence $Z$ is a subring of $R.$
\medskip

\noindent
(3)\quad
We already saw in the proof of part (1) that $3(x,y,z)=0,$
for all $x,y\in R,$ and $z\in Z,$ so $(x,y,z)=0,$ and $z\in C.$
Of course $C\subseteq Z.$ 
\end{proof}

The following lemma gives some properties of $R$ which we will require later.

\begin{lemma}\label{p of R}
Let $w,x,y,z\in R,$ and $n,n'\in N,$ then
\begin{enumerate}
\item
$[N,R]\subseteq N.$

\item
Nuclear elements commute with associators.

\item
$(w,n)(x,y,z)=-(x,n)(w,y,z).$ 
In particular, $(n',n)(x,y,z)=0.$

\item
If $N$ contains no non-zero elements which are zero divisors in $R,$ then $N$ is commutative.

\item
$(w,n)(w,n)(x,y,z)=0.$ 
\end{enumerate}
\end{lemma}
\begin{proof}
(1)\&(2)\quad
By \cite[equation (2), p.~129]{K2},
\begin{equation}\label{n(x,y,z)1}
(nx,y,z)=n(x,y,z).
\end{equation}
 Also, by the beginning of the proof of Lemma 4, p.~132 in \cite{K2},  
\begin{equation}\label{n(x,y,z)2}
(nx,y,z)  =  (x,y,z)n\quad\text{and}\quad (xn,y,z)  =  n(x,y,z).
\end{equation}
So $(nx,y,z)=(xn,y,z),$ hence (1) holds, and also $n(x,y,z)=(x,y,z)n,$ so (2) holds.
\medskip

\noindent
(3)\quad 
By \cite[equation (6), p.~131]{K2}, 
\begin{equation}\label{wx,n}
(wx,n)=w(x,n)+(w,n)x.
\end{equation}
Now  
\begin{alignat*}{3}
(w,n)(x,y,z)&=((w,n)x,y,z)& &\quad(\text{by (1) and equation \eqref{n(x,y,z)1}})\\
&=-(w(x,n),y,z)& &\quad\text{(by equation \eqref{wx,n} and (1)}) \\
&=-(x,n)(w,y,z)& &\quad\text{(by equation \eqref{n(x,y,z)2} and (1))}.
\end{alignat*}
Taking $w=n'\in N,$ we get the last statement of (3).
\medskip

\noindent
(4)\quad
This is immediate from the last statement of (3) and the hypothesis that $R$ is not associative.
\medskip

\noindent
(5)\quad
We have
\begin{alignat*}{3}
(w,n)(w,n)(x,y,z)&=-(w,n)(x,n)(w,y,z)& &\quad\text{(by (3))}\\
&=-(w,n)(w,y,z)(x,n)& &\quad\text{(by (2))}\\
&=(w,n)(y,w,z)(x,n)& &\quad\text{(associators are skew-symmetric)}\\
&=-(y,n)(w,w,z)(x,n)& &\quad\text{(by (3))}\\
&=0& &\quad \text{(because $R$ is alternative)}.
\end{alignat*}
\end{proof}

We also need the following three lemmas.

\begin{lemma}\label{lem commutators}
Assume that
\begin{itemize}
\item[(i)]
no non-zero commutator of $R$ is nilpotent;

\item[(ii)]
non-zero elements of the nucleus of $R$ are not divisors of zero in $R.$
\end{itemize}
Then a non-zero  commutator in $R$  is not a divisor of zero in $R.$
\end{lemma}
\begin{proof}
Let $v$ be a non-zero commutator. Then, by Facts \ref{facts}(1), $0\ne v^4\in N.$  
If $t\in R$ is such that $vt = 0,$ then since the subring generated by $v$ and $t$ is associative,
$v^4t   =  0,$ so $t = 0.$
\end{proof}

\begin{lemma}\label{lem nucleus}
Assume that
\begin{enumerate}
\item
no non-zero commutator of $R$ is nilpotent;

\item
$R$ has an identity;

\item
the nucleus of $R$ is a simple ring.
\end{enumerate}
Then   the nucleus of $R$ is a field, and
a non-zero  commutator in $R$  is not a divisor of zero in $R.$
\end{lemma}
\begin{proof}
First we show that $N$ is a field.
Recall from Lemma \ref{p of R} that nuclear elements commute with associators.  Set
\[
I:=\{n\in N\mid n(x,y,z)=0,\text{ for all }x,y,z\in R\}.
\] 
Then $(n'n)(x,y,z)=n'(n(x,y,z)=0,$ and $(nn')(x,y,z)=(x,y,z)(nn')=((x,y,z)n)n'=0,$ 
for $n\in I,\ n'\in N, \ x,y,z\in R.$ So $I$ is an ideal of $N.$ 
Hence $I=0,$ since $\one\notin I.$
However by Lemma \ref{p of R}(3), any commutator of elements in $N$ is in $I,$
so $N$ must be commutative. This plus simple is enough to make $N$ into a field.
\medskip

\noindent
Now
let $0\ne n\in N,$ such that $nt=0,$ for some $t\in R.$
Then $t=\one t=(n^{-1}n)t=n^{-1}(nt)=0.$  Note that
the identity of $R$ is the identity of $N.$
Hence $R$ satisfies the hypotheses of Lemma \ref{lem commutators},
and we are done.
\end{proof}

\begin{lemma}\label{commutative}
Assume that $2x=0$ implies $x=0,$ for all $x\in R.$
Then the  subring of $R$ generated by a triple $a,b,c$ 
of commuting elements is commutative.
\end{lemma}
\begin{proof}
We claim that for any triple of commuting elements $x,y,z\in R$ we have
\[\tag{$*$}
(xy,z)=0.
\]
Indeed, by \cite[equation (6), p.~131]{K2},
$(xy,z)=3(x,y,z).$ 
But also $(yx,z)=3(y,x,z)=-3(x,y,z).$  
Hence $2(3(x,y,z))=0,$ so
$3(x,y,z)=0.$ Hence $(xy,z)=0.$
A standard argument using $(*)$ and the length of
an element in the generators $\{a,b,c\}$ yields the lemma. 
\end{proof}
\medskip

\noindent
\begin{proof}[\bf Proof of Proposition \ref{prop 1.4}]
Part (1) follows from Lemmas \ref{lem commutators} and \ref{lem nucleus}, respectively.
Part (2) follows from Lemma \ref{lem nucleus}.  Suppose $R$ satisfies hypothesis
\ref{hyps 1.2}(T1).  By Lemma \ref{lem Z a subring}(3), $Z=C,$ so of course
hypothesis (2) of Theorem \ref{thm 1.1} holds.  Finally suppose $R$ satisfies hypothesis
\ref{hyps 1.2}(T2).  Let $x,y,z$ be as in hypothesis (2) of Theorem \ref{thm 1.1}.
By Lemma \ref{commutative}, the subring generated by $x,y,z$ is commutative and
hence associative, so hypothesis (2) of Theorem \ref{thm 1.1} holds.
\end{proof}

\section{The proof of Theorem 1.1}

In this section we prove Theorem \ref{thm 1.1}.  So we let $R$ 
be an alternative ring which is not associative.
We consider the following hypotheses. 

\[\tag{A}
\text{Non-zero commutators are not zero-divisors in R.}
\]
\[\tag{B}
\begin{aligned}
&\text{If $x\in R,$ and $y,z$ are in the commutative center of $R,$}\\
&\text{then  $(xy)z=x(yz)$.}
\end{aligned}
\]
Note that if hypothesis (B) holds in $R$ then $R$ is not commutative.

\begin{lemma}\label{lem 3.1}
Let $x,y,t\in R.$ Then 
\begin{enumerate}
\item
$((x,y),x,t)=(y,x,(x,t);$

\item
$(x,y)x=(x,yx);$

\item
if hypothesis (A) holds, and $xt=0,$ then $(x,yx)t=0.$
\end{enumerate}
\end{lemma}
\begin{proof}
(1)\quad
By   \cite[equation (4), p.~130]{K2} (and the fact that $R$ is alternative),
\[
(xy,x,t)=(y,x,t)x=(y,x,xt),
\]
and by \cite[equation (5), p.~130]{K2},
\[
(yx,x,t)=x(y,x,t)=(y,x,tx).
\]
Sustracting we get (1).
\medskip

\noindent
(2)\quad
This is immediate from Fact \ref{facts}(3).
\medskip

\noindent
(3)\quad
Assume that hypothesis (A) holds,
and that $xt=0.$  We claim that $(x,t)=0.$
Indeed, by Fact \ref{facts}(3), $(x,t)^2=(-tx)^2=t(xt)x=0,$
so, by hypothesis (A), $(x,t)=0.$

By part (1), $((x,y),x,t)=0.$ Since $xt=0,$ we see that
$((x,y)x)t=0,$ so (3) follows from (2).
\end{proof}

\begin{lemma}\label{lem R-Z nzd}
Assume that hypothesis (A) holds in $R.$
Let $x\in R,$ with $x\notin Z.$ Then $x$
is not a zero divisor in $R.$
\end{lemma}
\begin{proof}
Let $x\in R$ such that $x\notin Z.$ Then 
there exists $y\in R,$ with $(x,y)\ne 0.$  Suppose that
$xt=0,$ for some $t\in R.$  By Lemma \ref{lem 3.1}(3),
$(x,yx)t=0,$ and by Lemma \ref{lem 3.1}(2), $(x,yx)=(x,y)x.$
By hypothesis (A), $(x,yx)\ne 0,$ so by hypothesis (A)
again, $t=0.$
\end{proof}

\begin{lemma}\label{lem Z nzd}
Assume that hypotheses (A) and (B) hold in $R.$
Let $0\ne z\in Z,$ then $z$ is not a zero divisor in $R.$
\end{lemma}
\begin{proof}
Let $0\ne z\in Z,$ and suppose that $z$ is a divisor of zero. 
By Lemma \ref{lem R-Z nzd},
there exists non-zero $w\in Z,$ with $zw=0.$

Let $(x,y)$ be a non-zero commutator in $R.$
By Lemma \ref{lem R-Z nzd},
$xz\ne 0.$  By hypothesis (B),
$(xz)w=x(zw)=0.$ Hence by Lemma \ref{lem R-Z nzd} again,
$xz\in Z.$  

By equation \eqref{eq x,y,r}, 
\[
3(x,y,z)=(xy,z)-x(y,z)-(x,z)y=0
\]
But then also $3(x,z,y)=0,$ so by  equation \eqref{eq x,y,r} again,
\[
0=3(x,z,y)=(xz,y)-x(z,y)-(x,y)z.
\]
Since $(xz,y)=0=x(z,y),$ we see that $(x,y)z=0,$ contradicting hypothesis (A).
\end{proof}
\medskip

\begin{proof}[\bf Proof of Theorem \ref{thm 1.1}]\hfill
\medskip

Assume that hypotheses (A) and (B) hold in $R.$
By Lemmas \ref{lem R-Z nzd} and \ref{lem Z nzd},
$R$ contains no divisors of zero.  This show part (a).
 Part (b) follows
from part (a) and \cite[Theorem A]{BK}.
\end{proof}

\begin{remark}\label{rem 3.4}
As usual, a non-zero element $x\in R$ is a left zero divisor if there
exists a non-zero $y\in R$ 
such that $xy=0.$ Right zero divisors are similarly defined.  

In this paper, when we say that a non-zero element $r\in R$ is {\it not a zero divisor}
we mean that it is {\it either not a left zero divisor, or not a right zero divisor}.

Note that if a non-zero $r\in R$ is not a zero divisor, then $r^2\ne 0.$

Suppose now that we assume that  non-zero commutators in $R$ are not zero divisors.
Then $(x,y)^2\ne 0,$ for any non-zero commutator $(x,y).$  Now let $r\in R$ such
that $r$ is not a zero divisor.  Suppose that   $r$ is not a left zero divisor.
We show that $r$ is not a right zero divisor either.  Assume $sr=0,$ for some non-zero
$s\in R.$ Then $rs=(r,s)\ne 0.$ But  (since subrings generated
by two elements are associative), $(r,s)^2=r(sr)s=0,$
a contradiction.  Similarly any non-zero element which is not a right zero divisor
is also not a left zero divisor.

We thus see that in this paper, a non-zero element which is not a zero divisor (as defined
above) is {\it not a right and not a left zero divisor}.
\end{remark}

\subsection*{Acknowledgement}
We would like to thank Professor Holger Petersson for
remarks that improved the results of this paper.
We thank the referee for a thorough reading of the paper and for asking us to clarify certain points
in the paper.



\begin{thebibliography}{99}
\bibitem[1]{BK} R.~H.~Bruck, E.~Kleinfeld, {\it The structure of alternative division rings,} 
Proc.~Amer.~Math.~Soc.~{\bf 2} (1951), 878--890.

\bibitem[2]{Bl} F.~van der Blij, {\it History of the octaves,} Simon Stevin {\bf 34} (1961), 106-125.

\bibitem[3]{E} H.~D. Ebbinghaus et al., {\it Numbers,}   Graduate Texts in Mathematics, {\bf 123}. Readings in Mathematics. Springer-Verlag, New York, 1991.

\bibitem[4]{GR} E.~G.~Goodaire, D.~A.~Robinson, {\it Commutative alternative rings: a construction,} Comm.~Algebra {\bf 29} (2001),
no.~5, 1871--1882.

\bibitem[5]{K1} E.~Kleinfeld, {\it Simple alternative rings,} Ann.~of Math.~(2) {\bf 58} (1953), 544--547.

\bibitem[6]{K2} E.~Kleinfeld, {\it A Characterization of the Cayley Numbers,}
Math.~Assoc.~America Studies in Mathematics, Vol.~2, pp.~126--143, Prentice-Hall, Englewood Cliffs, N.~J., 1963.

\bibitem[7]{S} M.~F.~Smiley, {\it The radical of an alternative ring,}
Ann.~of Math.~(2) {\bf 49} (1948), 702--709.

\bibitem[8]{SV} T.~A.~Springer, F.~D.~Veldkamp, {\it Octonions, Jordan algebras and exceptional groups,} Springer Monographs in Mathematics. Springer-Verlag, Berlin, 2000. 

  
\end{thebibliography}
\end{document}